\documentclass[12pt,a4paper]{article}
\usepackage[margin=2cm]{geometry}
\usepackage{t1enc}
\usepackage[utf8]{inputenc}
\usepackage{amsthm,amsmath,amssymb}
\usepackage{graphicx}
\usepackage{enumerate}
\usepackage{hyperref}
\usepackage{bm}
\usepackage{comment}
\usepackage{amsfonts}
\usepackage{graphicx,caption}
\usepackage{bm}
\usepackage{amsmath, amsthm, amssymb}
\usepackage{graphicx}
\usepackage{hyperref}
 \usepackage{relsize}

\usepackage{bbm}

\theoremstyle{plain}
\usepackage{amsthm}
\makeatletter
\newcommand{\newreptheorem}[2]{\newtheorem*{rep@#1}{\rep@title}\newenvironment{rep#1}[1]{\def\rep@title{#2 \ref*{##1}}\begin{rep@#1}}{\end{rep@#1}}}
\makeatother

\newtheorem{theorem}{Theorem}
\newtheorem*{theorem-non}{Theorem}
\newtheorem*{non-lemma}{Lemma} 
\newtheorem{lemma}[theorem]{Lemma}
\newreptheorem{lemma}{Lemma}

\newtheorem{proposition}[theorem]{Proposition}

\theoremstyle{definition}
\newtheorem{remark}[theorem]{Remark}

\newcommand{\rwpc}{RGPC}
\newcommand{\rwpcs}{RGPCs}

\DeclareMathOperator{\range}{Im }
\DeclareMathOperator{\mm}{mm}

\newcommand{\Addresses}{{
  \bigskip
  \footnotesize

  Andr\'as M\'esz\'aros, 

\textsc{Central European University, Budapest \&}\par\nopagebreak
    
\textsc{Alfr\'ed R\'enyi Institute of Mathematics, Budapest}\par\nopagebreak
  \textit{E-mail address:}: \texttt{Meszaros\_Andras@phd.ceu.edu}

}}

\DeclareMathOperator{\rang}{Rank}
\title{Matchings on trees and the adjacency matrix:\\A determinantal viewpoint}
\author{Andr\'as M\'esz\'aros}
 \begin{document}

\maketitle

\begin{abstract}

Let $G$ be a finite tree. For any matching $M$ of $G$, let $U(M)$ be the set of vertices uncovered by $M$. Let $\mathcal{M}_G$ be a uniform random maximum size matching of $G$. 
In this paper, we analyze the structure of $U(\mathcal{M}_G)$. We first show that $U(\mathcal{M}_G)$ is a determinantal process. We also show that  for most vertices of $G$, the process $U(\mathcal{M}_G)$ in a small neighborhood of that vertex can be  well approximated  based on a somewhat larger neighborhood of the same vertex. Then we show that the normalized  Shannon entropy of $U(\mathcal{M}_G)$ can be also well approximated using  the local structure of $G$. In other words, in the realm of trees, the normalized  Shannon entropy of $U(\mathcal{M}_G)$ -- that is, the normalized logarithm  of the number of maximum size matchings of $G$ -- is a Benjamini-Schramm continuous parameter.

We show that $U(\mathcal{M}_G)$ is a determinantal process through establishing a new connection  between $U(\mathcal{M}_G)$ and the adjacency matrix of $G$. This result sheds a new light on the well-known  fact that on a tree, the number of vertices uncovered by a maximum size matching is equal to the nullity of the adjacency matrix.   

Some of the proofs are based on the well established method of introducing a new perturbative parameter, which we call temperature, and then define the positive temperature analogue of  $\mathcal{M}_G$, the so called monomer-dimer model, and let the temperature go to zero.



\end{abstract}




\section{Introduction}

First let us recall the notion of determintal measures. Let $V$ be a finite set. Let $P=(p_{ij})_{i,j\in V}$ be an orthogonal projection matrix. Then there is a unique probability measure $\nu_P$ on the subsets of $V$ with the following property. For any finite subset $F$ of $V$, we have
\[\nu_P(\{X|F\subseteq X\subseteq V\})=\det (p_{ij})_{i,j\in F}.\]
This unique probability measure is called the determinantal probability measure corresponding to $P$. A determinantal process corresponding to $P$ is a random subset of $V$ with law $\nu_P$.

Given a finite graph $G$, let $\bar{P}_G$ be the orthogonal projection to the kernel of the adjacency matrix of $G$. For any matching $M$ of $G$, let $U(M)$ be the set of vertices uncovered by $M$. Let $\mathcal{M}_G$ be a uniform random maximum size matching of $G$. 
\begin{theorem}\label{thmintdet}
Let $G$ be a finite tree, then $U(\mathcal{M}_G)$ is a determinantal process corresponding~to~$\bar{P}_G$.
\end{theorem} 

We need a few further notations. Given a vertex $v$ of a graph $G$, let $b_G^0(v)\in \mathbb{R}^{V(G)}$ be the characteristic vector of the set of neighbors of $v$. For a vertex $o$ of $G$, and a positive integer~$r$, let $\Pi_{G,o,r}$ be the orthogonal projection to the subspace of $\mathbb{R}^{V(G)}$ generated by the vectors $(b_G^0(v))_{v\in V(B_r(G,o))}$.\footnote{As usual,  $B_r(G,o)$ is the subgraph of $G$ induced by the vertices which are at most distance $r$ from $o$.} Finally, we let $\bar{P}_{G,o,r}=I-\Pi_{G,o,r}$. Note that if $r$ is at least the diameter of~$G$, then $\bar{P}_{G,o,r}=\bar{P}_G$.

 Our next theorem provides us a local approximation of the process $U(\mathcal{M}_G)$ for trees.

\begin{theorem}\label{thmlocal}
For any $D,r<\infty$ and  $\varepsilon>0$, we have an $R$ with the following property. Consider any finite tree $G$ with maximum degree at most $D$. Then there is a set of exceptional vertices $V_{ex}$ such that $|V_{ex}|\le \varepsilon|V(G)|$. For any vertex $o\not\in V_{ex}$, if we consider $U(\mathcal{M}_G)$ and the determintal processes corresponding to $\bar{P}_{G,o,R}$, and take their marginals in the window $B_r(G,o)$, then these marginals have total variation distance at most $\varepsilon$. 
\end{theorem}

In fact, we will prove the stronger statement that for any graph $G$ with maximum degree at most $D$, the determintal process corresponding to $\bar{P}_G$ can be approximated like above. However, if $G$ is not a tree, then the combinatorial interpretation of this process is not clear.  

Note that the marginal of the determintal processes corresponding $\bar{P}_{G,o,R}$ in the window $B_r(G,o)$ only depends on the $R+1$-neighborhood of $o$. Thus, for most points $o$, the process $U(\mathcal{M}_G)$ can be well approximated in the $r$-neighborhood of $o$ by only looking at the $R+1$-neighborhood of $o$.

Although at first sight, it might seem that the marginal of $\nu_{\bar{P}_{G,o,R}}$ in the ball $B_{R}(G,o)$ should be the same as the law of $U(\mathcal{M}_{B_R(G,o)})$, the situation is a bit more complicated. Let $(S,T)$ be the proper two coloring of $G$ for which $o\in S$. Let $R_{even}$ be the largest even integer which is not larger than $R+1$, and let $R_{odd}$ be the largest odd integer which is not larger than $R+1$. Let $U_{even}\subseteq V(B_{R_{even}}(G,o))$ and $U_{odd}\subseteq V(B_{R_{odd}}(G,o))$ be independent random sets such that $U_{even}$ has the same law as $U(\mathcal{M}_{B_{R_{even}}(G,o)})$ and  $U_{odd}$ has the same law as $U(\mathcal{M}_{B_{R_{odd}}(G,o)})$. Let $U$ have the law $\nu_{\bar{P}_{G,o,R}}$. Then $U\cap V(B_{R+1}(G,o))$ has the same law as $(U_{even}\cap S)\cup (U_{odd}\cap T)$.

A few questions arise naturally:
\begin{itemize}
\item Can the conclusion of Theorem \ref{thmlocal} hold for all vertices? The answer is no, as the following example shows. Let $G$ be an infinite rooted tree with root $o$, where vertices at even distance from the root have one child, and vertices at odd distance from the root have two children. Let $G_n=B_n(G,o)$. Then, as we calculate in Subsection \ref{pelda}, we have
\begin{equation}\label{peldeq1}
\lim_{n\to\infty}\mathbb{P}(o\in U(\mathcal{M}_{G_{2n}}))=\frac{1}{2},
\end{equation}
and
\begin{equation}\label{peldeq2}
\lim_{n\to\infty}\mathbb{P}(o\in U(\mathcal{M}_{G_{2n+1}}))=0,
\end{equation}
which shows that the conclusion of Theorem \ref{thmlocal} can not hold for all vertices.
\item Can we have a theorem similar to Theorem \ref{thmlocal} for $\mathcal{M}_G$ instead of $U(\mathcal{M}_G)$? The answer is again no. The reader might convince themself by following the hint given in Figure \ref{figure1}.\footnote{In fact, path graphs already provide a counter example. However, in that case, the neighborhoods have automorphisms. One can suggest a version of Theorem \ref{thmlocal} for $\mathcal{M}_G$, where we only consider the neighborhoods up to isomorphism. In that case, path graphs do not provide a   counter example anymore.}

\begin{figure}[h!]
  \captionsetup{width=.8\linewidth}
  \includegraphics[width=\textwidth]{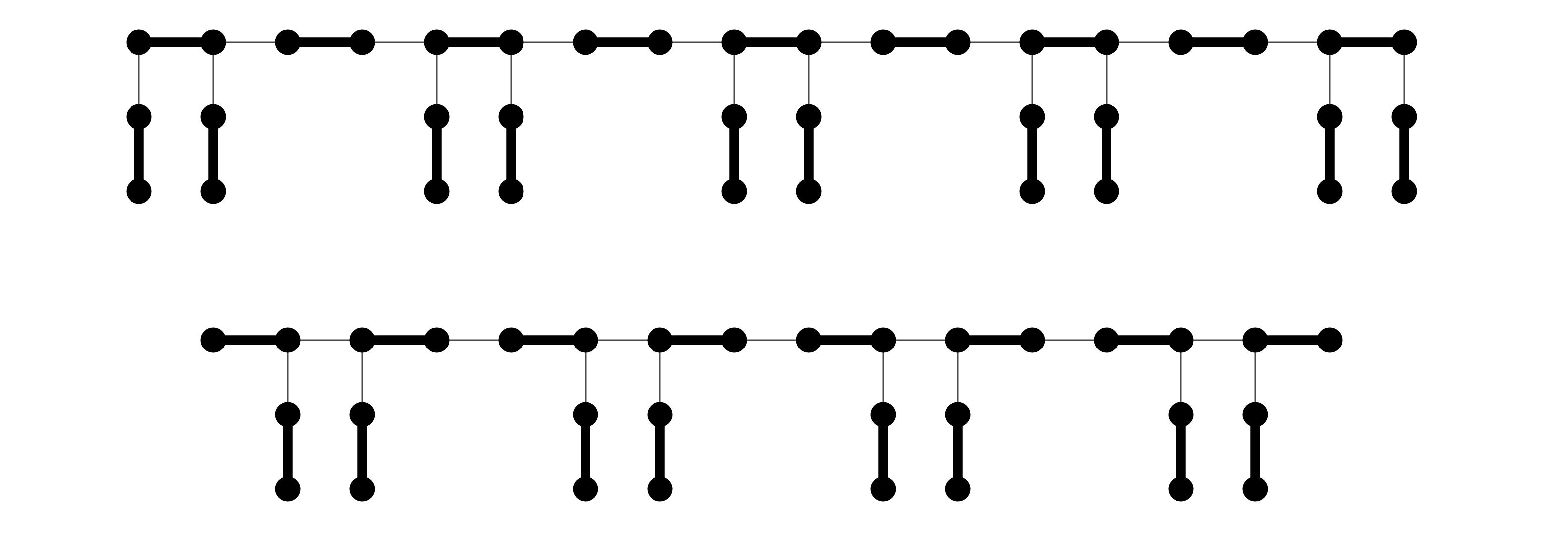}
  \caption{Two trees with similar local structures. Both graphs have a unique perfect matching, but locally these matchings look different.}\label{figure1}
\end{figure}
\item Can we omit the condition in Theorem \ref{thmlocal} that $G$ is a tree? We do not want to require that the approximation is done with a determinantal process, we only ask for a process that only depends on the $R+1$-neighborhood of the given vertex. To our knowledge, this question is open. The special case when $r=0$, that is, when we are only interested in the one vertex marginals, were answered by  Bordenave, Lelarge and Salez \cite{bls13}. In this case, the marginals can be indeed approximated. The discussion right after Theorem \ref{thmlocal} suggest that this question might be considerably easier for bipartite graphs.
\end{itemize}

%

%

Given a finite graph $G$, let $\mm(G)$ be the number of maximum size matchings of $G$. Our next theorem is the following.

\begin{theorem}\label{theoremFO}
Let $G_1,G_2,\dots$ be a Benjamini-Schramm convergent sequence of finite trees with maximum degree at most $D$. Then
\[\lim_{n\to\infty}\frac{\log \mm(G_n)}{|V(G_n)|}\]
exists.
\end{theorem}

The notion of Benjamini-Schramm convergence was introduced in \cite{benjamini2001recurrence}, see also \cite{aldous2007processes}. We recall it in Subsection \ref{BSconv}. Several graph parameters are known to be continuous with respect to Benjamini-Schramm convergence, for example: the normalized size of the maximum size matching \cite{elek2010borel,nguyen2008constant,bls13,abert2016matchings}, the normalized logarithm of the number of spanning trees \cite{lyons2005asymptotic}, the normalized rank of the adjacency matrix \cite{abert2013benjamini}.

Note that without the assumption that the graphs $G_i$ are trees, the limit in Theorem \ref{theoremFO} might not exist, even if the sequence converges to an amenable graph like $\mathbb{Z}^2$. We can see this by comparing the results of \cite{kasteleyn1961statistics,temperley1961dimer,elkies1992alternating}. See also \cite{abert2016matchings}, for an example of a Benjamini-Schramm convergent sequence of bipartite $d$-regular graphs such that the limit above does not exist.  However, if we restrict our attention to vertex transitive bipartite graphs, the limit above exists for convergent graph sequences, as it was proved by Csikv\'ari \cite{csikvari2016matchings}. The limit also exists for a sequence of bipartite $d$-regular large girth graphs \cite{abert2016matchings}.  

The first step in the proof of Theorem \ref{theoremFO} is the following simple observation. 
\begin{proposition}\label{propuniq}
Let $G$ be a finite tree. Then any matching $M$ of $G$ can be uniquely reconstructed from $U(M)$.
\end{proposition} 

It follows from Proposition \ref{propuniq} that $\log\mm(G)=H(U(\mathcal{M}_G))$, where $H$ denotes the Shannon entropy. 

Next we show that in the settings of Theorem \ref{theoremFO}, the pairs $(G_n,\bar{P}_{G_n})$ converge in a certain topology defined in \cite{meszaros2019limiting}. In fact, this statement is just a more abstract version of Theorem~\ref{thmlocal}, and again this is true for general convergent graph sequences not just for trees. Then it follows from the continuity results of the paper \cite{meszaros2019limiting}, that the normalized Shannon entropies of the determinantal processes corresponding to the orthogonal projections $\bar{P}_{G_n}$ converge. In other words, 
\[\frac{H(U(\mathcal{M}_n))}{|V(G_n)|}=\frac{\log \mm(G_n)}{|V(G_n)|}\]
converges, where $\mathcal{M}_n$ is a uniform random maximum size matching of $G_n$.

We also note that for various subgraphs of $\mathbb{Z}^2$, there are really strong results about the statistics of perfect matchings \cite{kasteleyn1961statistics,temperley1961dimer,elkies1992alternating,mh1,mh2,mh3,mh4}. Also, the positive temperature analogues of the questions above, that is, questions about the monomer-dimer models,  are also well understood even for general graphs not just for trees \cite{abert2015matching,bls13}.

Besides the already mentioned results
\begin{itemize}
\item We also obtain an analogue of Theorem \ref{thmintdet} for positive temperature Boltzmann random matchings.
\item We  extend  these results to infinite tress with uniformly bounded degrees. 
\item We also show that the entries of the corresponding projection matrices can be expressed with the help of the solution of a certain fixed point equation.
\item As an example, we give an explicit formula for the limit in Theorem \ref{theoremFO} in the special case of the balls of the $d$-regular infinite tree.
\end{itemize}
\textbf{Acknowledgements.} The author is grateful to Mikl\'os Ab\'ert and P\'eter  Csikv\'ari for their comments. The author was partially
supported by the ERC Consolidator Grant 648017. 
\section{Definitions and statements of the results}
\subsection{Benjamini-Schramm convergence}\label{BSconv}

Fix a finite degree bound $D$. Throughout the paper, we only consider graphs where the maximum degree is at most $D$.

 A \emph{rooted graph} is a pair $(G,o)$, where $G$ is a (possibly infinite) connected graph with uniform degree bound $D$, and $o$ is a distinguished vertex of $G$ called the root. Given two rooted graphs $(G_1,o_1)$ and $(G_2,o_2)$, their distance is defined as the infimum over all $\varepsilon>0$ such that for $r=\lfloor \varepsilon^{-1}\rfloor$, we have a root preserving isomorphism between $B_r(G_1,o_1)$ and $B_r(G_2,o_2)$. Here $B_r(G,o)$ denotes the rooted graph induced by the vertices of $G$ which are  at most distance $r$ from the root $o$. Let $\mathcal{G}_D$ be the set of isomorphism classes of rooted graphs. Endowed with the above defined metric, $\mathcal{G}_D$ is a compact metric space. This also gives us a measurable structure on $\mathcal{G}_D$. A sequence of random rooted graph $(G_n,o_n)$ Benjamini-Schramm converges to a random rooted graph $(G,o)$, if the distribution of $(G_n,o_n)$ converges to that of $(G,o)$ in the weak* topology.

Every finite graph $G$ can be turned into a random rooted graph $U(G)=(G_o,o)$ where $o$ is a uniform random vertex of $G$, and $G_o$ is the connected component of $o$ in the graph $G$. A sequence of finite graphs $G_n$ Benjamini-Schramm converges to the random rooted graph $(G,o)$ if $U(G_n)$ Benjamini-Schramm converges to $(G,o)$.     
 
\subsection{Preliminaries on determinantal measures}

We repeat the definition from the Introduction in slightly more general setting. Let $V$ be a finite or countably infinite set. Let $H$ be a closed subspace of the Hilbert-space $\ell^2(V)$. Let $P=(p_{ij})_{i,j\in V}$ be the matrix of the orthogonal projection to $H$ in the standard bases. Then there is a unique probability measure $\nu_H$ on the subsets of $V$ with the following property. For any finite subset $F$ of $V$, we have
\[\nu_H(\{X|F\subseteq X\subseteq V\})=\det (p_{ij})_{i,j\in F}.\]
This unique probability measure is called the determinantal probability measure corresponding to $H$( or $P$). A determinantal process corresponding to $H$( or $P$) is a random subset of $V$ with law $\nu_H$.

  If $V$ is finite, we have an alternative way to describe this measure. Let $B$ be a matrix, where the columns are indexed with $V$, moreover, the rows of $B$ form a bases of $H$. For $X\subseteq V$, let $B[X]$ be the submatrix of $B$ determined by the columns corresponding to the elements of $X$.   
\begin{samepage}
\begin{lemma}\label{altdef}
Let $r$ be the dimension of $H$. Consider  $X\subseteq V$. 
\begin{itemize}
\item If $|X|=r$, then 
\[\nu_H(\{X\})=\frac{|\det B[X]|^2}{\det (BB^T)}.\]
\item If $|X|\neq r$, then $\nu_H(\{X\})=0$.
\end{itemize}
\end{lemma}
\end{samepage}
For the proof and more information on determinatal measures, see the work of Lyons \cite{lyons}.

\subsection{Boltzmann random matchings}

For a finite graph $G$, let $\mathbb{M}(G)$ be the set of matching of $G$. Let $z>0$ be a positive real parameter called the temperature. Let $\mathcal{M}_G^z$ be a random matching of $G$ such that for any $M\in \mathbb{M}(G)$, we have

\[\mathbb{P}(\mathcal{M}_G^z=M)=\frac{z^{|V(G)|-2|M|}}{P_G(z)},\] where $P_G(z)=\sum_{M\in\mathbb{M}(G)} z^{|V(G)|-2M}$. The random matching $\mathcal{M}_G^z$ is called the Boltzmann random matching at temperature $z$.

Now we extend these definitions to a countably infinite graph $G$ with uniformly bounded degrees. Let $V_1\subseteq V_2\subseteq \dots$ be a sequence of finite subsets of $V(G)$ such that $\cup_{n=1}^\infty V_n=V(G)$. Let $G_n$ be the finite subgraph of $G$ induced by $V_n$. We call the sequence $(G_n)$ an exhaustion of~$G$. Since $E(G_n)\subseteq E(G)$, we can consider $\mathcal{M}_{G_n}^z$ as a random matching of $G$.

\begin{theorem}[Bordenave, Lelarge and Salez \cite{bls13}]
Let $G$ be a countably infinite graph with uniformly bounded degrees, and let $(G_n)$ be an exhaustion of $G$. Then $\mathcal{M}_{G_n}^z$ converge in law to a random matching $\mathcal{M}_G^z$. The law of  $\mathcal{M}_G^z$ does not depend on the chosen exhaustion.
\end{theorem}

We call the random matching $\mathcal{M}_G^z$ provided by the previous lemma the Boltzmann random matching at temperature $z$.

Recall that for a matching $M$, the set of vertices uncovered by $M$ is denoted by $U(M)$.  
The next theorem provides the existence of the zero temperate limit of $U(\mathcal{M}_G^z)$. 

\begin{theorem}[Bordenave, Lelarge and Salez \cite{bls13}]
Let $G$ be a finite or countably infinite graph with uniformly bounded degrees. Then $U(\mathcal{M}_G^z)$ converge in law to a random subset $\mathcal{U}_G$ of $V(G)$ as $z\to 0$.
\end{theorem}
\begin{remark}\label{remarkf1}
For a finite $G$, let $\mathcal{M}_G$ a uniform random maximum size matching of $G$. Then $\mathcal{M}_G^z$ converges to $\mathcal{M}_G$ in law as $z\to 0$. In particular, $\mathcal{U}_G$ has the same law as $U(\mathcal{M}_G)$. However, for infinite graphs, it is not clear whether $\mathcal{M}_G^z$ converges in law or not. 
\end{remark}

\subsection{Determinantal processes obtained from Boltzmann random matchings on trees}

In this subsection, we restrict our attention to trees with uniformly bounded degrees. Every tree is a bipartite graph. We will consider trees with a fixed proper two coloring. To emphasize this, we will use the term bipartite tree, by which we mean a tuple $(G,S,T)$, where $G$ is a tree, $S$ and $T$ are the two color classes of a proper two coloring of $G$.

Given a matching $M$ of the bipartite tree $(G,S,T)$, let $\Delta(M)$ be the symmetric difference of $U(M)$ and $T$. Or equivalently, let $\Delta(M)$ be the symmetric difference of $V(M)$ and $S$, where $V(M)$ is the set of vertices covered by $M$. For a finite tree, $|\Delta(M)|=|S|$ for any matching.

For $x\in V(G)$, and any real $z$, let $b_G^z(x)\in \ell^2(V(G))$ be the vector \[b_G^z(x)=z\chi_x+\sum_{y\sim x} \chi_y,\]
where $\sum_{y\sim x}$ denotes a summation over the neighbors $y$ of $x$, and $\chi_y$ is the characteristic vector of $y$.  

Let $R_G^z$ be the closed subspace of $\ell^2(V)$ generated by the vectors $\left(b_G^z(s)\right)_{s\in S}$.  Let $P_G^z$ be the orthogonal projection to $R_G^z$.

Note that these definitions depend on the choice of the proper two coloring of $G$. Whenever we need to emphasize this, we will write $R_{G,S,T}^z$ and $P_{G,S,T}^z$ in place of $R_{G}^z$ and $P_{G}^z$.
    
For a finite tree $G$, it is more convenient to describe $R_G^z$ as follows. Let $H_G$ be a matrix where the rows are indexed by $S$ and the columns are indexed by $T$, and each entry is $1$ if the corresponding element of $S\times T$ is an edge of $G$, and $0$ otherwise. Let $B_G^z$ be the (block) matrix 
\begin{equation}\label{BGzdef}
B_G^z=(zI\quad H_G)
\end{equation}
indexed by $S\times (S\cup T)=S\times V(G)$. Then $R_G^z$ is the row space of $B_G^z$.

We will prove the following theorem.
\begin{theorem}\label{postempdet}
Let $(G,S,T)$ be a (possibly infinite) bipartite tree with uniformly bounded degrees. Then $\Delta(\mathcal{M}_G^z)$ is the determinantal process corresponding to $R_G^z$. 
\end{theorem} 

Taking a zero temperature limit, with some additional work we will be able to deduce the following theorem.
\begin{theorem}\label{thmzerotdet}
Let $G$ be a (possibly infinite)  tree with uniformly bounded degrees. Then $\mathcal{U}_G$ is the determinantal process corresponding to the kernel of the adjacency operator of $G$.
\end{theorem} 

By Remark \ref{remarkf1}, we see that this theorem includes Theorem \ref{thmintdet} as a special case.

\subsection{A more general version of Theorem \ref{thmlocal}}

We will prove the following  more general version of Theorem \ref{thmlocal}.
\begin{theorem}\label{thmlocalgen}
For any $D,r<\infty$ and  $\varepsilon>0$, we have an $R$ with the following property. Consider any finite graph $G$ with maximum degree at most $D$. Then there is set of exceptional vertices $V_{ex}$ such that $|V_{ex}|\le \varepsilon|V(G)|$. For any vertex $o\not\in V_{ex}$, if we  consider the determintal processes corresponding to $\bar{P}_G$ and $\bar{P}_{G,o,R}$, and take their marginal in the window $B_r(G,o)$, then these marginals have total variation distance at most $\varepsilon$. 
\end{theorem}

For a tree $G$, we know that $U(\mathcal{M}_G)$ is the determintal processes corresponding to $\bar{P}_G$. Therefore, Theorem \ref{thmlocal} follows easily. We emphasize that Theorem \ref{thmlocalgen} is true for any graph~$G$, but if $G$ is not a tree, then the determinantal processes corresponding to $\bar{P}_G$ has no clear combinatorial relevance.

For a graph $G$, let $\Pi_G$ be the orthogonal projection to subspace of $\mathbb{R}^{V(G)}$ generated by the vectors $(b^0_G(v))_{v\in V(G)}$, that is, $\Pi_G$ is the orthogonal projection to rowspace of the adjacency matrix of $G$. We also have $\Pi_G=I-\bar{P}_{G}$.   Theorem \ref{thmlocalgen} is an easy consequence of the following lemma.

\begin{lemma}\label{lemma12}
For any $D,r<\infty$ and $\varepsilon_2>0$, there is an $R_2$ such that for any $R\ge R_2$ and for any finite graph $G$ with maximum degree at most $D$, the following holds. There is exceptional subset of the vertices $V_2\subset V(G)$ such that $|V_2|\le\varepsilon_2 |V(G)|$, and for any $o\in V(G)\backslash V_2$ and $u,v\in B_r(G,o)$, we have
\[\left|\langle\Pi_G\chi_u,\chi_v\rangle-\langle\Pi_{G,o,R}\chi_u,\chi_v\rangle\right|<\varepsilon_2,\] 
or equivalently,
\[\left|\langle\bar{P}_G\chi_u,\chi_v\rangle-\langle\bar{P}_{G,o,R}\chi_u,\chi_v\rangle\right|<\varepsilon_2.\]
\end{lemma}

\subsection{The space of rooted graph-positive-contractions}
First, we recall a few definitions from \cite{meszaros2019limiting}. 

As usual, fix a degree bound $D$. A \emph{rooted graph-positive-contraction} (\rwpc) is a triple $(G,o,T)$, where $(G,o)$ is a rooted graph and $T$ is a positive contraction on $\ell^2(V(G))$, that is, $T$ is a positive  semidefinite self-adjoint operator on $\ell^2(V(G))$ with norm at most $1$.  
Given two \rwpcs\  $(G_1,o_1,T_1)$ and $(G_2,o_2,T_2)$ their distance  is 
defined as the infimum over all $\varepsilon>0$ such that for $r=\lfloor \varepsilon^{-1}\rfloor$ there is a root preserving graph isomorphism $\psi$ from $B_r(G_1,o_1)$ to $B_r(G_2,o_2)$ with the property that 
\begin{equation*}
|\langle T_1\chi_v, \chi_w\rangle-\langle T_2\chi_{\psi (v)}, \chi_{\psi (w)}\rangle|<\varepsilon
\end{equation*} 
for every $v,w\in V(B_r(G_1,o_1))$. Recall that $\chi_v$ is the characteristic vector of $v$. 
Two \rwpcs\  $(G_1,o_1,T_1)$ and $(G_2,o_2,T_2)$ are called isomorphic if their distance is $0$, or equivalently if there is a root preserving graph isomorphism $\psi$ from $(G_1,o_1)$ to $(G_2,o_2)$ such that \[\langle T_1\chi_v,\chi_w\rangle=\langle T_2\chi_{\psi (v)}, \chi_{\psi (w)}\rangle\] for every $v,w\in V(G_1)$. Let $\mathcal{\rwpc}$ be the set of isomorphism classes of \rwpcs. One can prove that $\mathcal{\rwpc}$ is a compact metric space with the above defined distance. Let $\mathcal{P}(\mathcal{\rwpc})$ be the set of probability measures on $\mathcal{\rwpc}$ endowed with the weak* topology, this is again a compact space. Often it will be more convenient to consider  an element $\mathcal{P}(\mathcal{\rwpc})$ as a random \rwpc.

A \emph{finite graph-positive-contraction} is a pair $(G,T)$, where $G$ is finite graph with degrees at most $D$, and $T$ is a positive contraction on $\ell^2(V(G))$. It can be turned into a random \rwpc\ 
 \[U(G,T)=(G_o,o,T_o)\] by choosing $o$ as a uniform random vertex of $G$.

A sequence of finite graph-positive-contractions $(G_1,T_1),(G_2,T_2),\dots$ is Benjamini-Schramm converging to a random \rwpc\ $(G,o,T)$, if the sequence $U(G_n,T_n)$ converges in law to $(G,o,T)$.

The main result of the paper \cite{meszaros2019limiting} is the following theorem.
\begin{theorem}(\cite[Theorem 2.4]{meszaros2019limiting})\label{entdetcont}
Let $(G_n,P_n)$ be a sequence of finite graph-positive-contractions Benjamini-Schramm converging to a random \rwpc\ $(G,o,P)$. Assume that $P_1,P_2,\dots$ are orthogonal projections, and $P$ is an orthogonal projection with probability $1$. Let $X_n$ be the determinantal process corresponding to $P_n$. Then  
\[\lim_{n\to\infty} \frac{H(X_n)}{|V(G_n)|}\]
exists. 
\end{theorem}

To apply the theorem above, we need to prove the following theorem already mentioned in the Introduction. This theorem can be view as a more abstract form Lemma \ref{lemma12}. 

\begin{theorem}\label{projconv}
Let $G_1,G_2,...$ be a sequence of finite graphs Benjamini-Schramm converging to a random rooted graph $(G,o)$.  Then the sequence $(G_n,\bar{P}_{G_n})$ Benjamini-Schramm converges to $(G,o,\bar{P}_G)$. 
\end{theorem}

Combining this theorem with Theorem \ref{thmintdet}, Theorem \ref{entdetcont} and  the fact that $\log \mm(G)=H(\mathcal{U}_G)$ for any finite tree $G$, we obtain Theorem \ref{theoremFO}. Note that,  we only need to use Theorem~\ref{projconv} in the special case where every graph $G_i$ is a tree, but the Theorem \ref{projconv} is true for arbitrary graphs.

\begin{remark}\label{remarklim}
The results in \cite{meszaros2019limiting} also give us a formula for the limit in Theorem~\ref{theoremFO}. Let us consider a rooted tree $(G,o)$. For a vertex $v\in V(G)$, let $I(v)$ be the indicator of the event that $v\in \mathcal{U}_G$. Furthermore, let $\ell$ be a $[0,1]$ labeling of the vertices of $G$. Let 
\[\bar{h}(G,o,\ell)=H(I(o)|\{I(v)|\ell(v)<\ell(o)\}),\footnote{Here, $H$ is the conditional Shannon-entropy.}\]
and let 
\[\bar{h}(G,o)=\mathbb{E}\bar{h}(G,o,\ell),\]
where the expectation is over an i.i.d. uniform random $[0,1]$ labeling of the vertices. 

Then if the sequence of finite trees $G_1,G_2,\dots$ Benjamini-Schramm converges to the random rooted graph $(G,o)$, then we have 
\[\lim_{n\to\infty}\frac{\log \mm(G_n)}{|V(G_n)|}=\mathbb{E}\bar{h}(G,o),\]
where the expectation is over the random choice of $(G,o)$. 
\end{remark}

\subsection{Description of the projection matrices $P_G^z$ and $\bar{P}_G$}

This subsection contains a few complementary results that are not needed for the proof of Theorem \ref{thmintdet}, Theorem \ref{thmlocal} and Theorem \ref{theoremFO}.   We will give a description of the entries of the  projection matrices $P_G^z$ and $\bar{P}_G$. Note that the diagonal entries were already determined by Bordenave, Lelarge and Salez \cite{bls11,bls13}.

Let $(G,S,T)$ be a possibly infinite bipartite tree with uniformly bounded degrees. 
Let us fix a root $o\in V$. Then it makes sense to speak about the children of a given vertex. (A vertex $y$ is a child of a vertex $x$, if $y$ is neighbor of $x$, and $y$ is not on the unique path from $o$ to $x$.)  For a vertex $x$ the notation $\sum_{y\succ x}$ means summation over the children $y$ of $x$. 

Let us consider the following system of equations in the variables  $(m_x^z)_{x\in V}$: 
\begin{align}\label{rekur}
m^z_x=\frac{z^2}{z^2+\sum_{y\succ x} m^z_y}\qquad\qquad(x\in V).
\end{align} 

Note that $m_x^z=1$ for any leaf $x\neq o$. For a finite tree $G$, the other values can be easily computed by moving from the leaves towards the root. In particular, the system of equation above  has a unique solution, which is in $[0,1]^V$. We denote this solution by $m_{G,o,x}^z$.

It turns out that this is also true for any infinite tree $G$ like above.

Let $G_n$ be the subgraph induced by the vertices which are at most distance $n$ from the root~$o$.
\begin{lemma}[Bordenave, Lelarge, Salez \cite{bls13}]\label{uniq}\label{fixedpointptemp}
For a tree $G$ with uniformly bounded degrees and  any $z\neq 0$ the limit
\[m_x^z=\lim_{n\to\infty} m_{G_n,o,x}^z\]
exists for all $x\in V$. Moreover, the vector $(m_x^z)$ is the unique solution in $[0,1]^V$ of the system of equations given in \eqref{rekur}.
\end{lemma}

For $x\in V$, let $o=v_0,v_1,v_2,\dots,v_k=x$ be the unique path from $o$ to $x$. 
Let
\[
w_x^z=z^{-k}\prod_{i=0}^k m_{v_i}^z.
\] 

Given a vertex $x$, let $\ell(x)$ be its distance form the root. 

\begin{lemma}\label{lemmamatrix}
Let $z\neq 0$. If $o\in S$, then
\[\langle P_G^z \chi_o, \chi_x\rangle =
\begin{cases}
(-1)^{\lfloor\ell(x)/2\rfloor} w_x^z &\text{ if }x\in S,\\
(-1)^{\lfloor(\ell(x)-1)/2\rfloor}  w_x^z &\text{ if }x\in T.
\end{cases}
\] 
If $o\in T$, then
\[\langle P_G^z \chi_o, \chi_x\rangle=
\begin{cases}
1-m^z_o &\text{ if }x=o,\\
(-1)^{\lfloor(\ell(x)-1)/2\rfloor} w_x^z &\text{ if }o\neq x\in T,\\
(-1)^{\lfloor\ell(x)/2\rfloor}  w_x^z &\text{ if }x\in S.
\end{cases}\]
\end{lemma}

\begin{remark}
Combining the lemma above with  Theorem \ref{postempdet}, we see that $m_o^z$ is the probability that $o$ is uncovered by $\mathcal{M}_G^z$.
\end{remark}
Now we state the zero temperature version of Lemma \ref{fixedpointptemp}.

\begin{lemma}[Bordenave, Lelarge, Salez \cite{bls13}]
For each $x\in V$, the limit
\[m_x=\lim_{z\to 0} m_x^z\]
exists. Moreover, $(m_x)_{x\in V}$ is the largest solution in $[0,1]^V$ of the recursion 
\begin{equation}\label{zerorek}
m_x=\frac{1}{1+\sum_{y\succ x}\left(\sum_{u\succ y} m_u\right)^{-1}},
\end{equation}
with the convention $0^{-1}=\infty$ and $\infty^{-1}=0$.
\end{lemma}

\
 For a vertex $x$ such that $x$ is at even distance from the root $o$, let $o=v_0,v_1,\dots,v_{2k}=x$, be the unique path from $o$ to $x$. We define
\[w_x=m_o\prod_{i=1}^{k} \frac{m_{v_{2i}}}{\sum_{y\succ v_{2i-1} m_y}},\]
with the convention that $0/0=0$, that is, $w_x=0$ whenever $m_{v_{2i}}=0$ for some $i$.

\begin{theorem}\label{kerprojd}
Let $\bar{P}_G$ be the orthogonal projection to the kernel of the adjacency operator of $G$. Then
\[\langle \bar{P}_G \chi_o, \chi_x\rangle=
\begin{cases}
(-1)^{\lfloor(\ell(x)+1)/2\rfloor} w_x &\text{ if $x$ is at even distance from $o$},\\
0  &\text{ if $x$ is at odd distance from $o$}.
\end{cases}\]
\end{theorem}

\begin{remark}\label{remarkmo}
Combining the lemma above with Theorem \ref{thmzerotdet}, we see that $m_o=\mathbb{P}(o\in \mathcal{U}_G)$. 
\end{remark}

\subsection{An example: Balls of the $d$-regular tree}

Fix $d\ge 3$. Let $\mathbb{T}_d$ be the $d$-regular infinite tree, and let $o$ be any vertex of it. Then the sequence of finite graphs $G_n=B_n(\mathbb{T}_d,o)$ is Benjamini-Schramm convergent. The limit of this sequence is the so called \emph{$d$-canopy tree}, see for example \cite[Lemma 2.8]{dembo2010gibbs}. For this particular sequence of finite trees, we can give an explicit formula for the limit in Theorem \ref{theoremFO}.  

\begin{theorem}\label{thmgomb}
Let $G_n$ be like above, then
\begin{equation}\label{canopy}
\lim_{n\to\infty} \frac{\log\mm(G_n)}{|V(G_n)|}=\frac{\log (d-1)}{d} + (d-2)^2\sum_{\ell=2}^{\infty} (d-1)^{-2\ell}\log \ell.
\end{equation}
\end{theorem}
Note that by Theorem \ref{theoremFO}, equation \eqref{canopy} also holds for any sequence Benjamini-Schramm converging to the $d$-canopy tree.

\section{Determinantal processes obtained from Boltzmann random matchings on trees}

\subsection{The proof of Proposition \ref{propuniq}}

Let $M_1$ and $M_2$ be two matchings such that $U(M_1)=U(M_2)$. Let $D$ be the symmetric difference of $M_1$ and $M_2$. It is easy to see that $D$ must be a vertex disjoint union of cycles. The graph $G$ is a tree, so it has no cycles. Therefore, $D$ must be empty, that is, $M_1=M_2$.

\subsection{The proof of Theorem \ref{postempdet}}
First we prove the statement for a finite $G$. Recall that the matrix $B_G^z$ was defined in \eqref{BGzdef}, and for $X\subseteq V(G)$, the submatrix of $B_G^z$   determined by the columns corresponding to the elements of $X$ is denoted by $B_G^z[X]$.

We can obtain Theorem \ref{postempdet} by combining Lemma~\ref{altdef} and the following proposition.

\begin{proposition}\label{propBX}
Let $X$ be a subset of $V(G)$ such that $|X|=|S|$. 
\begin{itemize}
\item If there is no matching $M\in\mathbb{M}(G)$ such that $\Delta(M)=X$, then
\[\left|\det B_G^z[X]\right|^2=0.\]
\item Otherwise, let $M$ be the unique matching such that $\Delta(M)=X$. Then
\[\left|\det B_G^z[X]\right|^2=z^{2|S\cap X|}=z^{2(|S|-|M|)}=z^{|S|-|T|}\cdot z^{|V(G)|-2|M|}.\]
\end{itemize}

\end{proposition} 
\begin{proof}
We can order rows and columns of $B_G^z[X]$ with without changing the value $|\det B_G^z[X]|^2$. By choosing an appropriate ordering of rows and columns, $B_G^z[X]$ will be an upper triangular block matrix, that is,
\[B_G^z[X]=
\begin{pmatrix}
zI_{S\cap X}& F\\
0 & H_{G_0}
\end{pmatrix}, 
\]
where $I_{S\cap X}$ is the identity matrix, where the rows and columns are indexed by $S\cap X$, $H_{G_0}$ the matrix corresponding to the edge set of the subgraph $G_0$ of $G$ induced by the vertices in symmetric difference of $X$ and $S$
, the matrix $F$ will not be interesting for us now. 

We have
\[|\det B_G^z [X]|^2=|\det z I_{S\cap X}|^2\cdot|\det H_{G_0}|^2=z^{2|S\cap X|}|\det H_{G_0}|^2.\]

Observe that in the Leibniz formula for the determinant of $H_{G_0}$, the non-zero terms correspond to the perfect matchings in $G_0$, in other words, to the matchings $M$ in $G$ such that $\Delta(M)=X$. From Proposition \ref{propuniq}, there is  at most one such matching. If there is a matching like that, then $|\det H_{G_0}|=1$, otherwise $|\det H_{G_0}|=0$. The statement follows.     
\end{proof}
This concludes the proof for finite trees. Now assume that $G$ is infinite.

We will need the following proposition. The proof is straightforward.

\begin{proposition}\label{propSOT}
Let $P_1, P_2, \dots$ be orthogonal projections in $\ell^2(V)$ such that they converge to an orthogonal projection $P$ in the strong operator topology. Then the determinantal measures corresponding to the projections $P_i$ converge weakly  to the  determinantal measure corresponding to $P$.   
\end{proposition}
Pick a vertex $o\in T$. Consider the exhaustion $(G_k)$ of the graph $G$, where $G_k=B_{2k}(G,o)$. For any positive integer $k$, we have that $b_{G_k}^z(s)=b_G^z(s)$ for any $s\in  V(G_k)\cap S$. Therefore, $R_{G_k}^z$ is subspace of $R_G^z$. Moreover, $(R_{G_k}^z)_{k\ge 1}$ is an increasing sequence of subspaces such that their union is dense in $R_G^z$. Thus, it follows that $\lim_{k\to\infty} P_{G_{k}}^z=P_G^z$ in the strong operator topology. By Proposition \ref{propSOT}, the determinantal process corresponding to $R_G^z$ is the weak limit of the determinatal processes corresponding to $R_{G_{k}}^z$. It follows from the already established finite case, that the determinatal processes corresponding to $R_{G_{k}}^z$ is $\Delta(\mathcal{M}_{G_k}^z)$. However, it follows directly from the definitions that $\Delta(\mathcal{M}_{G}^z)$ is  the weak limit of $\Delta(\mathcal{M}_{G_k}^z)$. The statement follows.       

\subsection{The proof of Theorem \ref{thmzerotdet}}
We start by the following lemma.
\begin{lemma}\label{ortho}
For any bipartite tree $(G,S,T)$, the subspaces $R_{G,S,T}^z$ and $R_{G,T,S}^{-z}$ are orthogonal complements of each other. 
\end{lemma}
\begin{proof}
Note that $b_G^z(s)$ is orthogonal to $b_G^{-z}(t)$ for any $s\in S$ and $t\in T$. This imply that the subspaces $R_{G,S,T}^z$ and  $R_{G,T,S}^{-z}$ are orthogonal. To finish the proof, we need to prove that if $v\in \left(R_{G,S,T}^z+R_{G,T,S}^{-z}\right)^\bot$, then $v=0$. Given a $v\in \left(R_{G,S,T}^z+R_{G,T,S}^{-z}\right)^\bot$, let us define the vector $w\in \ell^2(V(G))$ by
\[w(x)=\begin{cases}
v(x)&\text{if }x\in S,\\
v(x)i&\text{if }x\in T,
\end{cases}
\] 
where $i$ is the imaginary unit. Let $A$ be the adjacency operator of $G$. It is easy to check that $Aw=-z i w$. Since $A$ is self-adjoint, all the eigenvalues of $A$ are real, so this is only possible if $w=v=0$. 
\end{proof}

\begin{lemma}\label{strongconv}
Let $(G,S,T)$ be a bipartite tree, then for any $v\in \ell^2(T)\subseteq \ell^2(V(G))$, we have
\[\lim_{z\to 0}P_{G,S,T}^z v=P_{G,S,T}^0 v\]
in norm.
\end{lemma}
\begin{proof}
Let $P_T$ be the orthogonal projection from $\ell^2(V)$ to $\ell^2(T)$. Note that the image of $R_{G,S,T}^z$ under the projection $P_T$ is contained in  $R_{G,S,T}^0$. Therefore,
\begin{align}\label{ineq222}\|v-P_{G,S,T}^z v\|^2&=\min_{p\in R_{G,S,T}^z} \|v-p\|^2\ge \min_{p\in R_{G,S,T}^z} \|P_T(v-p)\|^2\nonumber\\ &\ge \min_{p\in R_{G,S,T}^0} \|v-p\|^2=\|v-P_{G,S,T}^0 v\|^2.
\end{align}
  Pick any $\varepsilon>0$, then there is a finite set $F\subseteq S$, and coefficients $(\alpha_f)_{f\in F}$, such that for $\|P_{G,S,T}^0 v-\sum_{f\in F} \alpha_f b_G^0(f)\|<\varepsilon$. It is straightforward to see that 
\[\|P_{G,S,T}^0 v-\sum_{f\in F} \alpha_f b_G^z(f)\|<2\varepsilon\]
for any small enough $z$. Then for small enough $z$, we have 

\begin{align}\label{ineq1}
\|P_{G,S,T}^z v-P_{G,S,T}^0 v\|&\le \|P_{G,S,T}^z v-\sum_{f\in F} \alpha_f b_G^z(f)\|+\|\sum_{f\in F} \alpha_f b_G^z(f)-P_{G,S,T}^0 v\|\\&\le  \|P_{G,S,T}^z v-\sum_{f\in F} \alpha_f b_G^z(f)\|+2\varepsilon.\nonumber
\end{align}
Since $P_{G,S,T}^z v-\sum_{f\in F} \alpha_f b_G^z(f)$ is orthogonal to $v-P_{G,S,T}^z v$, for small enough $z$, we have
\begin{multline*}
\|P_{G,S,T}^z v-\sum_{f\in F} \alpha_f b_G^z(f)\|^2=\| v-\sum_{f\in F} \alpha_f b_G^z(f)\|^2- \|v-P_{G,S,T}^z v\|^2\\=
\| v-\sum_{f\in F} \alpha_f b_G^z(f)\|^2-\|v-P_{G,S,T}^0 v\|^2+\|v-P_{G,S,T}^0 v\|^2- \|v-P_{G,S,T}^z v\|^2\\=
\langle 2v-\sum_{f\in F} \alpha_f b_G^z(f)-P_{G,S,T}^0 v,P_{G,S,T}^0v -\sum_{f\in F} \alpha_f b_G^z(f) \rangle +\|v-P_{G,S,T}^0 v\|^2- \|v-P_{G,S,T}^z v\|^2\\\le
\| 2v-\sum_{f\in F} \alpha_f b_G^z(f)-P_{G,S,T}^0 v\|\cdot \|P_{G,S,T}^0v -\sum_{f\in F} \alpha_f b_G^z(f) \|\\\le
(2\|v\|+2\|P_{G,S,T}^0 v\|+2\varepsilon)2\varepsilon,
\end{multline*}
where in the second to last inequality we used \eqref{ineq222}.
Inserting this into \eqref{ineq1} we obtain that \[\|P_{G,S,T}^z v-P_{G,S,T}^0 v\|\le 2\varepsilon + \sqrt{(2\|v\|+2\|P_{G,S,T}^0 v\|+2\varepsilon)2\varepsilon},\]
for any small enough $z$. Tending to zero with $\varepsilon$, we get that $\lim_{z\to 0} P_{G,S,T}^z v=P_{G,S,T}^0 v$ in norm.
\end{proof}

We need tow simple facts about determinantal processes.

\begin{proposition}(\cite[
Corollary 5.3.]{lyons})\label{propdet21}
Let $E$ be any countable set,  let $P$ be an orthogonal projection on $\ell^2(E)$, and let $B$ be the determinantal process corresponding to $P$. Then $E\backslash B$ is a  determinantal process corresponding to $I-P$. 
\end{proposition} 
\begin{proposition}\label{propdet22}
Let $E_1$ and $E_2$ be two disjoint countable sets, let $P_1$ and $P_2$ be orthogonal projections on $\ell^2(E_1)$ and $\ell^2(E_2)$, respectively. Let $B_1$ and $B_2$ be independent determinantal processes corresponding to $P_1$ and $P_2$, respectively. Then $B_1\cup B_2$ is a determinantal process corresponding to $P_1\oplus P_2$.   
\end{proposition}

Let $A$ be the adjacency operator of $G$. Let $P_1$ be the ortogonal projection from $\ell^2(S)$ to $\ker A\cap \ell^2(S)$, and let  $P_2$ be the ortogonal projection from $\ell^2(T)$ to $\ker A\cap \ell^2(T)$. It is easy to see that $\bar{P}_G=P_1\oplus P_2$. 

The next lemma is an easy consequence of Lemma \ref{strongconv}.
\begin{lemma}\label{strongconv2}
For any bipartite tree $(G,S,T)$, the projections $P_{G,S,T}^z$ converge to $P_1\oplus (I- P_2)$ in the strong operator topology as $z$ tends to $0$.\footnote{At first sight it may be surprising that the limit is not $P_{G,S,T}^0$.} 
\end{lemma}
\begin{proof}
First, consider a $v\in \ell^2(T)\subseteq \ell^2(V)$. Then Lemma \ref{strongconv} gives that us 
\[\lim_{z\to 0} P_{G,S,T}^zv=P_{G,S,T}^0 v=(I-P_2)v.\]
Now consider a $v\in \ell^2(S)\subseteq \ell^2(V)$, then we have
\[\lim_{z\to 0} P_{G,S,T}^z v=\lim_{z\to 0} (I- P_{G,T,S}^{-z}) v=(I-P_{G,T,S}^0) v=P_1v,\]  
where the first equality follows from Lemma \ref{ortho}, the second one follows from Lemma \ref{strongconv}.

Since $\ell^2(V)=\ell^2(S)\oplus \ell^2(T)$, the statement follows.
\end{proof}

Now we are ready to prove Theorem \ref{thmzerotdet}.

Let $B_S$ and $B_T$ be independent determinantal processes corresponding to $P_1$ and $P_2$, respectively. Combining Proposition \ref{propdet21} and Proposition \ref{propdet22}, we get that $B'=B_S\cup(T\backslash B_T)$ is a determinantal process corresponding to $P_1\oplus(I-P_2)$. By Lemma \ref{strongconv2} and Proposition \ref{propSOT}, we see that $\Delta(\mathcal{M}_G^z)$ converges in law to $B'$ as $z\to 0$. Comparing the definition of $\Delta(\mathcal{M}_G^z)$ and $U(\mathcal{M}_G^z)$, we see that $U(\mathcal{M}_G^z)$ converges in law to $B_S\cup B_T$. Thus, by definition the law $B_S\cup B_T$ must be the same as that of $\mathcal{U}_G$. Since $B_S\cup B_T$ is the determinantal process corresponding to $\bar{P}_G=P_1\oplus P_2$ by Proposition \ref{propdet22}, the theorem follows. 

\section{The local structure of the projection matrices $\bar{P}_G$}

\subsection{The proof of Theorem \ref{thmlocalgen}}

For a graph $G$, let $\Pi_G$ be the orthogonal projection to the closed subspace of $\ell^2(V(G))$ generated by the vectors $(b^0_G(v))_{v\in V(G)}$. Given a rooted graph $(G,o)$ and a positive integer $R$, let $\Pi_{G,o,R}$ be the the orthogonal projection to the subspace of $\ell^2(V(G))$ generated by the vectors $(b^0_G(v))_{v\in V(B_R(G,o))}$. 

Let us consider a finite graph $G$, let $\rang(G)$ be rank of its adjacency matrix. It is a simple fact from linear algebra that
\[\rang(G)=\sum_{o\in V(G)}\langle \Pi_G\chi_o,\chi_o\rangle.\]
Thus, it is reasonable the expect that for large enough $R$, the quantity
\[\rang_R(G)=\sum_{o\in V(G)}\langle \Pi_{G,o,R}\chi_o,\chi_o\rangle\]
gives us a good approximation of $\rang(G)$. The next lemma shows that this is indeed true.

\begin{lemma}\label{rankapprox}
For any $D<\infty$ and $\varepsilon>0$, there is an $R$ such that for any finite graph $G$ with maximum degree at most $D$, we have
\[\rang_R(G)\le \rang(G) \le \rang_R(G)+\varepsilon |V(G)|.\]
\end{lemma}
\begin{proof}
The inequality $\rang_R(G)\le \rang(G)$ follows from the fact that for any $o\in V(G)$, the image of $\Pi_{G,o,R}$ is contained in the image of $\Pi_G$.

We prove the rest of the lemma by contradiction. Assume that for some $D<\infty$ and $\varepsilon>0$, we have a sequence of finite graphs $(H_R)_{R\in \mathbb{N}}$ such that $H_R$ has maximum degree at most $D$ and $\rang(H_R)>\rang_R(H_R)+\varepsilon |V(H_R)|$. From the compactness of the space $\mathcal{G}_D$, it follows that there is a subsequence $R_1<R_2<\dots$ of the positive integers such that $(H_{R_n})$ is Benjamini-Schramm convergent. Let $G_n=H_{R_n}$, and let the random rooted graph $(G,o)$ be the limit of the sequence $(G_n)$. It follows from the result of Ab\'ert, Thom and Vir\'ag \cite{abert2013benjamini} that 
\[\lim_{n\to\infty} \frac{\rang(G_n)}{|V(G)|}=\mathbb{E}\langle \Pi_G \chi_o,\chi_o\rangle.\]

In particular, for any large enough $n$, we have
\[\left|\frac{\rang(G_n)}{|V(G)|}-\mathbb{E}\langle \Pi_G \chi_o,\chi_o\rangle \right|<\frac{\varepsilon}3.\]

Note that for any fixed rooted graph $(G',o')$, the projections $\Pi_{G',o',R}$ converge to $\Pi_{G'}$ in the strong operator topology. Combining this with the dominated convergence theorem, we obtain that
\[\lim_{R\to \infty} \mathbb{E} \langle \Pi_{G,o,R} \chi_o,\chi_o\rangle=\mathbb{E}\lim_{R\to \infty}  \langle \Pi_{G,o,R} \chi_o,\chi_o\rangle=\mathbb{E} \langle \Pi_G \chi_o,\chi_o\rangle.\]
Thus, for a large enough $R$, 
\[\left|\mathbb{E} \langle \Pi_G \chi_o,\chi_o\rangle-\mathbb{E}\langle \Pi_{G,o,R} \chi_o,\chi_o\rangle \right|<\frac{\varepsilon}3. \]
Fix such an $R$. It is straightforward to see that for a large enough $n$, we have 
\[\left|\mathbb{E} \langle \Pi_{G,o,R} \chi_o,\chi_o\rangle-\frac{\rang_R(G_n)}{|V(G_n)|}\right|<\frac{\varepsilon}3.\]
Thus, for any large enough $n$, we have
\[\rang(G_n)<\rang_R(G_n)+\varepsilon |V(G_n)|.\] 

Choose an $n$, such that $R_n>R$. Then
\[\rang(G_n)<\rang_R(G_n)+\varepsilon |V(G_n)|\le \rang_{R_n}(G_n)+\varepsilon |V(G_n)|,\]
where in the last step, we used the fact that $\rang_{r}(G_n)$ is monotone increasing in $r$. This gives us a contradiction.
\end{proof}

\begin{lemma}\label{rankapprox2}
For any $D<\infty$ and $\varepsilon_1>0$, there is an $R$ such that for any finite graph $G$ with maximum degree at most $D$, the following holds. There is exceptional subset of the vertices $V_1\subset V(G)$ such that $|V_1|\le\varepsilon_1 |V(G)|$, and for any $o\in V(G)\backslash V_1$, we have
\[\langle\Pi_G\chi_o,\chi_o\rangle-\langle\Pi_{G,o,R}\chi_o,\chi_o\rangle<\varepsilon_1.\] 
\end{lemma}
\begin{proof}
We apply Lemma \ref{rankapprox} for $D$ and $\varepsilon=\varepsilon_1^2$. Let $R$ be the constant provided by that lemma. Consider that any finite graph $G$ with maximum degree at most $D$. 

Note that $\langle\Pi_G\chi_o,\chi_o\rangle-\langle\Pi_{G,o,R}\chi_o,\chi_o\rangle\ge 0$ for any $o\in V(G)$,  because the image of $\Pi_{G,o,R}$ is contained in the image of  $\Pi_G$. Moreover, by the choice of $R$, we have
\[\sum_{o\in V(G)}\left( \langle\Pi_G\chi_o,\chi_o\rangle-\langle\Pi_{G,o,R}\chi_o,\chi_o\rangle\right)=\rang(G)-\rang_R(G)\le \varepsilon |V(G)|.\]
We set
\[V_1=\{o\in V(G)\quad |\quad \langle\Pi_G\chi_o,\chi_o\rangle-\langle\Pi_{G,o,R}\chi_o,\chi_o\rangle\ge \varepsilon_1\}. \]
Then, by Markov's inequality, we have
\[|V_1|\le \frac{\varepsilon|V(G)|}{\varepsilon_1}=\varepsilon_1|V(G)|.\]
\end{proof}

Now we are ready to prove Lemma \ref{lemma12}. We repeat the statement for the reader's convenience.
\begin{non-lemma}[Lemma \ref{lemma12} repeated]
For any $D,r<\infty$ and $\varepsilon_2>0$, there is an $R_2$ such that for any $R\ge R_2$ and for any finite graph $G$ with maximum degree at most $D$, the following holds. There is exceptional subset of the vertices $V_2\subset V(G)$ such that $|V_2|\le\varepsilon_2 |V(G)|$, and for any $o\in V(G)\backslash V_2$ and $u,v\in B_r(G,o)$, we have
\[\left|\langle\Pi_G\chi_u,\chi_v\rangle-\langle\Pi_{G,o,R}\chi_u,\chi_v\rangle\right|<\varepsilon_2,\] 
or equivalently,
\[\left|\langle\bar{P}_G\chi_u,\chi_v\rangle-\langle\bar{P}_{G,o,R}\chi_u,\chi_v\rangle\right|<\varepsilon_2.\]
\end{non-lemma}
\begin{proof}
We set $K=(r+1)D^r$. Note that any ball of radius $r$ in graph with maximum degree at most $D$ can contain at most $K$ vertices. We set $\varepsilon_1=\min\left(\frac{\varepsilon_2}{K},\varepsilon_2^2\right)$. Let $R_0$ be the constant provided by Lemma \ref{rankapprox2}. We set $R_2=R_0+r$. Let $R\ge R_2$. 

Consider that any finite graph $G$ with maximum degree at most $D$. Let $V_1\subset V(G)$ be the exceptional set provided by Lemma \ref{rankapprox2}.

 We set
\[V_2=\cup_{o\in V_1} V(B_r(G,o)).\]
Note that
\[|V_2|\le K|V_1|\le K\varepsilon_1 |V(G)|\le \varepsilon_2 |V(G)|.\]
If we choose any  $o\in V(G)\backslash V_2$ and $u,v\in B_r(G,o)$, then by the choice of $V_2$, we have that $u,v\in V(G)\backslash V_1$. Using that fact that the image of $\Pi_{G,o,R}$ is contained in the image of $\Pi_G$, we see that $\Pi_G-\Pi_{G,o,R}$ is an orthogonal projection. Therefore,
\[\|(\Pi_G-\Pi_{G,o,R})\chi_u\|^2=\langle (\Pi_G-\Pi_{G,o,R})\chi_u,\chi_u\rangle= \langle\Pi_G\chi_u,\chi_u\rangle-\langle\Pi_{G,o,R}\chi_u,\chi_u\rangle.\]
Now observe that $B_{R_0}(G,u)$ is contained in $B_{R}(G,o)$, thus the image of $\Pi_{G,u,R_0}$ is contained in $\Pi_{G,o,R}$. So  $\langle\Pi_{G,o,R}\chi_u,\chi_u\rangle\ge \langle \Pi_{G,u,R_0} \chi_u,\chi_u\rangle$. Therefore,
\[\|(\Pi_G-\Pi_{G,o,R_2})\chi_u\|^2= \langle\Pi_G\chi_u,\chi_u\rangle-\langle\Pi_{G,o,R}\chi_u,\chi_u\rangle\le \langle\Pi_G\chi_u,\chi_u\rangle-\langle\Pi_{G,u,R_0}\chi_u,\chi_u\rangle<\varepsilon_1,\]
since $u\in V(G)\backslash V_1$. Finally,
\[\left|\langle\Pi_G\chi_u,\chi_v\rangle-\langle\Pi_{G,o,R}\chi_u,\chi_v\rangle\right|=\left|\langle(\Pi_G-\Pi_{G,o,R})\chi_u,\chi_v\rangle\right|\le \|(\Pi_G-\Pi_{G,o,R})\chi_u\|<\sqrt{\varepsilon_1}\le\varepsilon_2 .\]
\end{proof}

We omit the proof of the following simple proposition.
\begin{proposition}
Fix a positive integer $n$, and an $\epsilon>0$. Then there is an $\delta>0$ with the following property. Let $A$ and $B$ two $n\times n$ matrices such that their entries are contained in  $[-1,1]$. Moreover, the entries of $A-B$ are contained in $[-\delta,\delta]$. Then $|\det A-\det B|<\varepsilon.$ 
\end{proposition}

Combining this proposition with Lemma \ref{lemma12}, we obtain Theorem~\ref{thmlocalgen}.  


\subsection{The proof of Theorem \ref{projconv}}

We need the following characterization of local weak convergence.
\begin{proposition}
\hfill
\begin{enumerate}[(i)]
\item The sequence $(G_n,o_n)$ of random rooted graphs converges to the random rooted graph $(G,o)$ if and only if the following holds. For any $r<\infty$ and $\varepsilon>0$, we have an $n_0$ such that for any $n>n_0$, we have a coupling of $(G_n,o_n)$ and $(G,o)$ such that with probability at least $1-\varepsilon$, the ball $B_r(G_n,o_n)$ is isomorphic to ball $B_r(G,o)$.
\item The sequence $(G_n,o_n,T_n)$ of random \rwpcs\ converges to the random \rwpc\ $(G,o,T)$ if and only if the following holds. For any $r<\infty$ and $\varepsilon>0$, we have an $n_0$ such that for any $n>n_0$, we have a coupling of $(G_n,o_n,T_n)$ and $(G,o,T)$ such that with probability at least $1-\varepsilon$, we have a rooted isomorphism $\psi$ from $B_r(G,o)$ to $B_r(G_n,o_n)$ such that for all $x,y\in V(B_r(G,o))$, we have
\[\left|\langle T\chi_x,\chi_y\rangle-\langle T_n\chi_{\psi(x)},\chi_{\psi(y)}\rangle\right|<\varepsilon.\]  
\end{enumerate}
\end{proposition} 
\begin{proof}
The statements follow easily from the fact that the Wasserstein metric induces the topology of the weak* convergence.
\end{proof}

Now we prove Theorem \ref{projconv}. Let $\Pi_n=\Pi_{G_n}$. The statement of the theorem is equivalent to that the sequence $(G_n,\Pi_n)$ converges to $(G,o,\Pi_G)$. It will be more convenient to prove this statement instead. 

Let $o_n$ be a uniform random vertex of $G_n$. With a slight abuse of notation, we will write $(G_n,o_n,\Pi_n)$ instead of the more precise $((G_n)_{o_n},o_n,\Pi_n)$. 
 
Pick any $r<\infty$ and $\varepsilon>0$. As $R$ tends to infinity, the images of $\Pi_{G,o,R}$ form an increasing  sequence of subspaces such that their union is a dense subset  of the image of $\Pi_G$. Therefore, $\Pi_{G,o,R}$ converge to $\Pi_G$ in the strong operator topology.  Therefore, there is an $R_1$ such that for any $R\ge R_1$ the following holds. Consider the limiting random rooted graph $(G,o)$. With probability $1-\varepsilon$, we have  
\begin{align*}
\left|\langle \Pi_G\chi_x,\chi_y\rangle-\langle \Pi_{G,o,R}\chi_{x},\chi_{y}\rangle\right|<\varepsilon && \forall x,y\in B_r(G,o).
\end{align*}

Let $R_2$ be the constant provided by Lemma \ref{lemma12}. Take any $R$ such that $R\ge \max(R_1,R_2)$. Then for any large enough $n$, we have a coupling of $(G_n,o_n)$ and $(G,o)$ such that with probability $1-\varepsilon$, an isomorphism between $B_{R+1}(G,o)$ and $B_{R+1}(G_n,o_n)$. Let $V_2\subset V(G_n)$ be the set exceptional vertices provided by Lemma \ref{lemma12}. Then the following event have probability at least $1-3\varepsilon$ (with respect to the joint distribution of $(G_n,o_n)$ and $(G,o)$):
\begin{itemize}
\item  We have \begin{align*}
\left|\langle \Pi_G\chi_x,\chi_y\rangle-\langle \Pi_{G,o,R}\chi_{x},\chi_{y}\rangle\right|<\varepsilon && \forall x,y\in B_r(G,o),
\end{align*}

\item also $o_n\in V(G_n)\backslash V_2$, 

\item furthermore, $B_{R+1}(G,o)$ and $B_{R+1}(G_n,o_n)$ are isomorphic.

\end{itemize} 

On this event, let $\bar{\psi}$ be a rooted isomorphism from $B_{R+1}(G,o)$ to $B_{R+1}(G_n,o_n)$, and let $\psi$ be the restriction of $\bar{\psi}$ to $B_r(G,o)$.

Then for any $x,y\in B_r(G,o)$, we have 
\begin{align*}
|\langle \Pi_{G}\chi_x,\chi_y\rangle&-\langle \Pi_{G_n}\chi_{\psi(x)},\chi_{\psi(y)}\rangle|\\
&\le\left|\langle \Pi_{G}\chi_x,\chi_y\rangle-\langle \Pi_{G,o,R}\chi_x,\chi_y\rangle\right|+\left|\langle \Pi_{G,o,R}\chi_x,\chi_y\rangle-\langle \Pi_{G_n,o_n,R}\chi_{\psi(x)},\chi_{\psi(y)}\rangle\right|\\&\qquad+\left|\langle \Pi_{G_n,o_n,R}\chi_{\psi(x)},\chi_{\psi(y)}\rangle- \langle\Pi_{G_n}\chi_{\psi(x)},\chi_{\psi(y)}\rangle\right|\\
&<\varepsilon+0+\varepsilon.
\end{align*}
Thus, the statement follows.

\section{Description of the projection matrices $P_G^z$ and $\bar{P}_G$}

\subsection{The proof of Lemma \ref{lemmamatrix}}
We consider $o$ as the root $G$. For any vertex $x$, let $V_x$ be the set of descendants of $x$ including $x$ itself, that is, a vertex $v$ is in $V_x$ if the unique path from $o$ to $v$ contains $x$. Let $S_x=S\cap V_x$, $T_x=T\cap V_x$. Let $G_x$ be the subgraph of $G$ induced by $V_x$.  Let  $b_x^z(s)=b_{G_x}^z(s)$, $R_x^z=R_{G_x,S_x,T_x}^z$ and let $P_x^z$ be the orthogonal projection to $R_x^z$.

Finally, let $h_x^z=\|P_x^z\chi_x\|^2=\langle P_x^z\chi_x,\chi_x\rangle$.
\begin{proposition}
If $x\in S$, then
\begin{equation}\label{fLkS0}
P_x^z \chi_x=\alpha \left(b_x^z(x)-\sum_{t\succ x} P_t^z \chi_t\right),
\end{equation}
where \[\alpha=\frac{z}{z^2+\sum_{t\succ x}(1-h_t^z)}.\]
 Moreover,
\[h_x^z=\|P_x^z \chi_x\|^2=\frac{z^2}{z^2+\sum_{t\succ x}(1-h_t^z)}.\] 

If $x\in T$, then
\[P_x^z \chi_x=\chi_x-\alpha \left(b_x^{-z}(x)-\sum_{s\succ x} (\chi_s-P_s^z \chi_s)\right)=(1+\alpha z)\chi_x-\alpha \sum_{s\succ x} P_s^z \chi_s ,\]
where \[\alpha=\frac{-z}{z^2+\sum_{s\succ x}h_s^z}.\] Moreover,
\[h_x^z=\|P_x^z \chi_x\|^2=1-\frac{z^2}{z^2+\sum_{s\succ x}h_s^z}.\] 
\end{proposition}

\begin{proof}
Fist assume that $x\in S$. Observe that  $R_t^z$ is a subspace of $R_x^z$ for any $t\succ x$. Therefore, it is clear that
\[\alpha \left(b_x^z(x)-\sum_{t\succ x} P_t^z \chi_t\right)\in R_x^z.\] 

Observe that
\[\chi_x-\alpha \left(b_x^z(x)-\sum_{t\succ x} P_t^z \chi_t\right)=(1-z\alpha) \chi_x-\alpha\sum_{t\succ x}(\chi_t-P_t^z\chi_t).\]
Take any $s\in S_x$ such that $s\neq x$. Let $r$ be the unique child of $x$ such that $s\in S_r$. Then
\[\langle (1-z\alpha) \chi_x-\alpha\sum_{t\succ x}(\chi_t-P_t^z\chi_t),b_x^z(s)\rangle=-\alpha\langle \chi_r-P_r^z\chi_r,b_r^z(s)\rangle=0.\]
Moreover,
\begin{align*}
\langle (1-z\alpha) \chi_x-\alpha\sum_{t\succ x}(\chi_t-&P_t^z\chi_t),b_x^z(x)\rangle\\&=\langle (1-z\alpha) \chi_x-\alpha\sum_{t\succ x}(\chi_t-P_t^z\chi_t),z\chi_x+\sum_{t\succ x}\chi_t\rangle\\
&=z(1-z\alpha)-\alpha\sum_{t\succ x} \langle \chi_t-P_t^z\chi_t,\chi_t\rangle\\
&=z(1-z\alpha)-\alpha\sum_{t\succ x}(1-h_t^z)\\
&=z-\alpha(z^2+\sum_{t\succ x}(1-h_t^z))=0.
\end{align*}

This proves that $P_x^z \chi_x=\alpha \left(b_x^z(x)-\sum_{t\succ x} P_t^z \chi_t\right)$. Furthermore,
\begin{align*}
h_x^z&=\|P_x^z \chi_x\|^2\\&=\alpha^2 \|b_x^z(x)\|^2+ \alpha^2\sum_{t\succ x}( \left \|P_t^z \chi_t\|^2-2\langle b_x^z(x),P_t^z \chi_t \rangle\right)\\&=\alpha^2 \|b_x^z(x)\|^2+ \alpha^2\sum_{t\succ x}( \left \|P_t^z \chi_t\|^2-2\langle \chi_t,P_t^z \chi_t \rangle\right)\\&=\alpha^2(z^2+\sum_{t\succ x} 1)-\alpha^2\sum_{t\succ x}h_t^z\\
&=\alpha^2(z^2+\sum_{t\succ x} 1)-\alpha^2\sum_{t\succ x}h_t^z\\&=\alpha^2 z^2+\alpha^2\sum_{t\succ x}\left(1- h_t^z\right)=\frac{z^2}{z^2+\sum_{t\succ x}(1-h_t^z)}. 
\end{align*}

 Now we prove the case when $x\in T$. Let $Q_x^z$ be the orthogonal projection to the subspace $R_{G_x,T_x,S_x}^z$. By Lemma \ref{ortho}, we see that $P_x^z=I-Q_x^{-z}$. Using the already established case, we see that
 \[Q_x^{-z} \chi_x=\alpha \left(b_x^{-z}(x)-\sum_{s\succ x} Q_s^{-z} \chi_s\right),\]
where 
 \[\alpha=\frac{-z}{z^2+\sum_{s\succ x}(1-\|Q_s^{-z}\chi_s\|^2)}=\frac{-z}{z^2+\sum_{s\succ x}h_s^{z}}.\]
 Then
 \begin{multline*}
P_x^z\chi_x=\chi_x-Q_x^{-z}\chi_x=\chi_x-\alpha\left(b_x^{-z}(x)-\sum_{s\succ x} Q_s^{-z} \chi_s\right)\\=\chi_x-\alpha\left(b_x^{-z}(x)-\sum_{s\succ x} (\chi_s-P_s^z\chi_s) \chi_s\right).
\end{multline*}

 Moreover, using again the already established case, we see that
 \begin{multline*}
\|P_x^z\chi_x\|^2=1-\|Q_x^{-z}\chi_x\|^2=1-\frac{z^2}{z^2+\sum_{s\succ x}(1-\|Q_s^{-z}\chi_s\|^2)}\\=1-\frac{z^2}{z^2+\sum_{s\succ x}\|P_s^{z}\chi_s\|^2}=1-\frac{z^2}{z^2+\sum_{s\succ x}h_s^z}.
\end{multline*}
\end{proof}

Let 
\[m_x^z=\begin{cases} 
h_x^z & \text{if }x\in S,\\
1-h_x^z &\text{if } x\in T.
\end{cases}
\]

It is clear from the previous lemma that
\begin{proposition}\label{proprekurz}
\[m_x^z=\frac{z^2}{z^2+\sum_{y\succ x} m_y^z}.\]
\end{proposition}

Thus $m_x^z$ is the unique solution of the system of equations given in Lemma \ref{fixedpointptemp}.


\begin{lemma}\label{twosT}
If $x\in T$, then
\[P_x^z\chi_x=\sum_{s\succ x}\frac{m_x^z m_s^z}{z^2}\left(b_G^z(s)-\sum_{t\succ s} P_t^z\chi_t\right).\]
\end{lemma}

\begin{proof}
Note that for \[\alpha=\frac{-z}{z^2+\sum_{s\succ x} h_s^z},\] we have
\[\alpha=\frac{-z}{z^2+\sum_{s\succ x} m_s^z}=-\frac{1}{z}\frac{1}{1+\sum_{s\succ x}(z^2+\sum_{t\succ s}m_t^z)^{-1}}.\]

Thus
\begin{align*}
P_x^z\chi_z&=(1+\alpha z)\chi_x-\alpha \sum_{s\succ x} P_s^z \chi_s\\&=\frac{\sum_{s\succ x}(z^2+\sum_{t\succ s}m_t^z)^{-1}}{1+\sum_{s\succ x}(z^2+\sum_{t\succ s}m_t^z)^{-1}}\chi_x\\&\qquad\qquad +\frac{1}{1+\sum_{s\succ x}(z^2+\sum_{t\succ s}m_t^z)^{-1}} \sum_{s\succ x}\left(z^2+\sum_{t\succ s} m_t^z\right)^{-1}\left(b_s^z(s)-\sum_{t\succ s} P_t^z\chi_t\right)\\&=\frac{1}{1+\sum_{s\succ x}(z^2+\sum_{t\succ s}m_t^z)^{-1}} \sum_{s\succ x}\left(z^2+\sum_{t\succ s} m_t^z\right)^{-1}\left(b_s^z(s)+\chi_x-\sum_{t\succ s} P_t^z\chi_t\right)\\&=\frac{1}{1+\sum_{s\succ x}(z^2+\sum_{t\succ s}m_t^z)^{-1}} \sum_{s\succ x}\left(z^2+\sum_{t\succ s} m_t^z\right)^{-1}\left(b_G^z(s)-\sum_{t\succ s} P_t^z\chi_t\right)\\&=m_x^z \sum_{s\succ x}\left(z^2+\sum_{t\succ s} m_t^z\right)^{-1}\left(b_G^z(s)-\sum_{t\succ s} P_t^z\chi_t\right)\\&=\sum_{s\succ x}\frac{m_x^z m_s^z}{z^2}\left(b_G^z(s)-\sum_{t\succ s} P_t^z\chi_t\right).
\end{align*}

\end{proof}

Let $L_k$ be the set of point at distance $k$ from the root. 

\begin{lemma}
If $o\in S$, then for any $k$, we have
\begin{equation}\label{fLkS}
P_G^z\chi_o=\sum_{i=0}^k \sum_{s\in L_{2i}} (-1)^{i} z^{-1}w^z_s b_G^z(s)-(-1)^k\sum_{s\in L_{2k}}z^{-1}w^z_s\sum_{t\succ s} P_t^z \chi_t.
\end{equation}
If $o\in T$, then for any $k$, we have  
\begin{equation}\label{fLkT}
P_G^z\chi_o=\sum_{i=0}^k \sum_{s\in L_{2i+1}} (-1)^{i} z^{-1}w^z_s b_G^z(s)-(-1)^k\sum_{s\in L_{2k+1}}z^{-1}w^z_s\sum_{t\succ s} P_t^z \chi_t.
\end{equation}
Moreover,
\begin{equation}\label{szomf}
\langle P_G^z\chi_o,\chi_x\rangle=
\begin{cases}
(-1)^{\lfloor \ell(x)/2 \rfloor}w^z_x & \text{if }x\in S,\\
\sum_{y\sim x} (-1)^{\lfloor \ell(y)/2 \rfloor} z^{-1}w^z_y &\text{if }x\in T.
\end{cases}
\end{equation}
\end{lemma}
\begin{proof}
If $k=0$ the formula in (\ref{fLkS}) reduces to (\ref{fLkS0}). Then we can prove the statement by induction using Lemma \ref{twosT}. The formula in (\ref{fLkT}) is proved similarly. The last statement is obtained by combining the previous two.
\end{proof}

Now we are ready to finish the proof Lemma \ref{lemmamatrix}.

If $x\in S$, then from (\ref{szomf}), we have $\langle P_G^z\chi_o,\chi_x\rangle=(-1)^{\lfloor \ell(x)/2 \rfloor}w^z_x$.

Let us consider $o\neq x\in T$. Let  $p(x)$ be the parent of $x$.  

Then from (\ref{szomf}) we have
\begin{align*}
\langle P_G^z\chi_o, \chi_x \rangle &= \sum_{y\sim x} (-1)^{\lfloor \ell(y)/2 \rfloor} z^{-1}w_y=(-1)^{\lfloor (\ell(x)-1)/2 \rfloor}w_{p(x)}z^{-1}\left(1-\frac{m_x}{z^2}\sum_{y\succ x} m_y\right)\\&=(-1)^{\lfloor (\ell(x)-1)/2 \rfloor}w_{p(x)}z^{-1}\left(1-\frac{\sum_{y\succ x} m_y}{z^2+\sum_{y\succ x} m_y}\right)\\&=(-1)^{\lfloor (\ell(x)-1)/2 \rfloor}w_{p(x)}z^{-1}m_x=(-1)^{\lfloor (\ell(x)-1)/2 \rfloor} w_x.
\end{align*}

Finally, if $o\in T$, then we have
\[\langle P_G^z\chi_o,\chi_o\rangle=\|P_G^z \chi_o\|=h^z_o=1-m_o^z.\]

\subsection{The proof of Theorem \ref{kerprojd}}

Consider the proper two coloring $(S,T)$ of $G$ such that $o\in T$.
 
Let $A$ be the adjacency matrix of $G$, then $\ker A$ is the orthogonal direct sum $\ker A\cap \ell^2(S)$ and $\ker A\cap \ell^2(T)$. Therefore, it is clear that
\[\langle P_G\chi_o,\chi_x\rangle = 0\]
for all $x\in T$, that is, for all $x$ which is at odd distance from the root.

Now consider the case that $x$ is at even distance from the root, that is $x\in T$. Let $o=v_0,v_1,\dots,v_{2k}=x$, be the unique path from $o$ to $x$. Note that \[\bar{P}_G\chi_o=(I-P_{G,S,T}^0)\chi_o=\lim_{z\to 0} (I-P_{G,S,T}^z)\chi_o,\]
where the last equality follows from Lemma \ref{strongconv}.
 
  First we assume that $m_{v_{2i}}> 0$ for all $0<i\le k$, then using Lemma \ref{lemmamatrix} and Proposition \ref{proprekurz}, we obtain that
\begin{align*}
\langle \bar{P}_G\chi_o,\chi_x\rangle&=\lim_{z\to 0} \langle (I-P_{G,S,T}^z)\chi_o,\chi_x\rangle \\&=\lim_{z\to 0} (-1)^{\lfloor(\ell(x)+1)/2\rfloor} w_x^z\\&=(-1)^{\lfloor(\ell(x)+1)/2\rfloor} \lim_{z\to 0} z^{-2k} \prod_{i=0}^{2k} m^z_{v_i}\\&=(-1)^{\lfloor(\ell(x)+1)/2\rfloor}\lim_{z\to 0} m_o^z \prod_{i=1}^{k} m^z_{v_{2i}} \frac{m^z_{v_{2i-1}}}{z^2}\\&=
(-1)^{\lfloor(\ell(x)+1)/2\rfloor} \lim_{z\to 0} m_o^z \prod_{i=1}^{k} \frac{m^z_{v_{2i}}}{z^2+\sum_{y\succ v_{2i-1}} m^z_y}\\&=(-1)^{\lfloor(\ell(x)+1)/2\rfloor}m_o\prod_{i=1}^{k} \frac{m_{v_{2i}}}{\sum_{y\succ v_{2i-1}} m_y}\\&= (-1)^{\lfloor(\ell(x)+1)/2\rfloor} w_x
\end{align*}

Now assume that $m_{v_{2i}}=0$ for some $i>0$.  We need to prove that  $\langle \bar{P}_G \chi_o,\chi_x\rangle =0$.   
Let $\hat{G}$ be the subgraph of $G$ induced by the set descendants $\hat{V}$ of $v_{2i}$ including $v_{2i}$. Let $\hat{S}=S\cap \hat{V}$ and $\hat{T}=T\cap \hat{V}$.  We consider $v_{2i}$ as the root of $\hat{G}$. Combining Lemma \ref{strongconv} with Lemma \ref{lemmamatrix}, we get
\[\|\bar{P}_{\hat{G}} \chi_{v_{2i}}\|^2=\lim_{z\to 0} \|(I-P_{\hat{G},\hat{S},\hat{T}}^z) \chi_{v_{2i}}\|^2=\lim_{z\to 0} m_{v_{2i}}^z=m_{v_{2i}}=0.\] 

Thus,
\[\chi_{v_{2i}}=(I-\bar{P}_{\hat{G}}) \chi_{v_{2i}}=P_{\hat{G},\hat{S},\hat{T}} \chi_{v_{2i}}.\]
So, $\chi_{v_{2i}}$ is in the subspace generated by $(b_{\hat{G}}^0(s))_{s\in \hat{S}}$. Since $b_{\hat{G}}^0(s)=b_{G}^0(s)$ for any $s\in \hat{S}$, we see that $\chi_{v_{2i}}$ is in $(\ker A)^\bot$. 
Let $P_{\hat{T}}$ the orthogonal projection to $\ell^2(\hat{T})$.  Since $x\in \hat{T}$, it is enough to prove that $P_{\hat{T}} \bar{P}_G \chi_o=0$. Assume that $P_{\hat{T}} \bar{P}_G \chi_o\neq 0$. We claim that that $P_{\hat{T}} \bar{P}_G \chi_o\in \ker A$. We need to prove that $P_{\hat{T}} \bar{P}_G \chi_o$ is orthogonal to $b_G^0(v)$ for any $v\in V(G)$. If $v\in T$, then $b_G^0(v)\in \ell^2(S)$, thus, $\langle P_{\hat{T}} \bar{P}_G \chi_o,b_G^0(v)\rangle=0$. If $v\in S$, we consider three cases
\begin{itemize}
\item If $v\in\hat{V}$, then  $\langle P_{\hat{T}}\bar{P}_G \chi_o, b_G^0(v)\rangle=\langle \bar{P}_G \chi_o, P_{\hat{T}} b_G^0(v)\rangle=\langle \bar{P}_G \chi_o,b_G^0(v)\rangle=0$, so the statement follows.
\item If $v\not\in \hat{V}$ and $v_{2i}$ is not a neighbor of $v$, then  $\langle P_{\bar{T}}\bar{P}_G \chi_o, b_G^0(v)\rangle=0$ because the two vectors have disjoint support. 
\item $v\not\in \hat{V}$ and $v_{2i}$ is  a neighbor of $v$, then \[\langle P_{\hat{T}}\bar{P}_G \chi_o, b_G^0(v)\rangle=\langle P_{\hat{T}}\bar{P}_G \chi_o,\chi_{v_{2i}}\rangle=\langle  \bar{P}_G\chi_o, P_{\hat{T}}\chi_{v_{2i}}\rangle=\langle  \bar{P}_G\chi_o, \chi_{v_{2i}}\rangle=0,\]
since $\chi_{v_{2i}}\in (\ker A)^\bot$.
\end{itemize}

Let $p=\bar{P}_G \chi_o-P_{\hat{T}}\bar{P}_G \chi_o$. Then $p\in \ker A$.
It is clear that $\|\chi_o-\bar{P}_G\chi_o\|>\|\chi_o-p\|$, which is contradiction because $h=\bar{P}_G\chi_o$ should be the unique minimizer of $\min_{h\in \ker A} \|\chi_o-h\|.$

\section{Examples}
\subsection{The number of maximum size matchings in balls of the $d$-regular tree}
In this section, we prove Theorem \ref{thmgomb}. Although Remark \ref{remarklim} provides us  a formula for the limit in Theorem \ref{theoremFO}, it seems very hard to evaluate this formula in practice. Thus, we proceed with a more direct approach.

Let $k=d-1$. Let $T_n$ be the complete $k$-ary tree of depth $n$. To be more specific, $T_0$ is a single vertex, and for $n\ge 1$, $T_n$ is a tree where we have one vertex of degree $k$ called the root, every leaf is at distance $n$ form the root, and vertices other than the root and the leaves have degree $k+1$.

One can easily prove that the sequence $(G_i)$ and the sequence $T_{2i+1}$ have the same Benjamini-Schramm limit. Therefore, by Theorem \ref{theoremFO}, we can work with the sequence $T_{2i+1}$ instead of $G_i$.
 
  Let $a_n=\mm(T_n)$ be the number of maximum size matchings of $T_n$. Let $m_n$ be the probability that the root of $T_n$ is uncovered by a uniform random maximum size matching of $T_n$. 

\begin{lemma}\label{uprob}
For any integer $i\ge 0$, we have
\[m_{2i}=\frac{1}{i+1},\]
and
\[m_{2i+1}=0.\]
\end{lemma} 
\begin{proof}
It is straightforward to check that $m_0=1$ and $m_1=0$. By Remark \ref{remarkmo} and equation~\eqref{zerorek}, for any $n\ge 2$, we have
\[m_n=\frac{1}{1+k(km_{n-2})^{-1}}=\frac{1}{1+m_{n-2}^{-1}}.\]
Therefore, the statement follows by induction.
\end{proof}

\begin{lemma}
For any $i>0$, we have
\[a_{2i}=(i+1)a_{2i-1}^k,\]
and 
\[a_{2i+1}=ka_{2i-1}^k a_{2i}^{k-1}.\]
Consequently,
\[a_{2i+1}=k(i+1)^{k-1}a_{2i-1}^{k^2}.\] 
\end{lemma}
\begin{proof}
If we delete the root of $T_{2i}$, then the resulting graph is the disjoint union of $k$ graphs isomorphic to $T_{2i-1}$. Thus, this smaller graph has $a_{2i-1}^k$ maximum size matchings. Since $m_{2i}>0$ by Lemma \ref{uprob}, these matchings are exactly the maximum size matchings of $T_{2i}$ which leave the root uncovered. Therefore, by Lemma \ref{uprob}, we have
\[a_{2i}=m_{2i}^{-1} a_{2i-1}^k=(i+1) a_{2i-1}^k.\] 

To prove the second statement, observe that by Lemma \ref{uprob}, every maximum size matching of $T_{2i+1}$ must cover the root. There are $k$ edges incident to the root. If we delete one of these edges together with its endpoints, then the resulting graph will be the disjoint union of $k$ trees isomorphic to $T_{2i-1}$ and $k-1$ trees isomorphic to $T_{2i}$. Thus, the statement follows easily.

We can obtain the third statement by combining the previous two. 
\end{proof}

Using the last statement of the lemma above, by induction, we obtain that for any $1\le h\le i$, we have 
\[\log a_{2i+1}=k^{2h}\log a_{2(i-h)+1}+\left(\sum_{j=0}^{h-1} k^{2j}\right)\log k+ (k-1)\sum_{j=0}^{h-1} k^{2j} \log(i+1-j).\]
Since $a_1=k$, with the choice of $h=i$, we obtain that
\begin{align*}
\log a_{2i+1}&=\left(\sum_{j=0}^{i} k^{2j}\right)\log k+(k-1)\sum_{j=0}^{i-1} k^{2j} \log(i+1-j) \\
&=\frac{k^{2(i+1)}-1}{k^2-1}\log k + (k-1)\sum_{\ell=2}^{i+1} k^{2(i+1-\ell)}\log \ell.
\end{align*} 

Note that $|V(T_{2i+1})|=\frac{k^{2(i+1)}-1}{k-1}$. Therefore,
\[\frac{\log a_{2i+1}}{|V(T_{2i+1})|}=\frac{\log k}{k+1} + (k-1)^2\frac{k^{2(i+1)}}{k^{2(i+1)}-1}\sum_{\ell=2}^{h+1} k^{-2\ell}\log \ell.\]

Thus, it follows easily that
\[\lim_{i\to \infty}\frac{\log a_{2i+1}}{|V(T_{2i+1})|}=\frac{\log k}{k+1} + (k-1)^2\sum_{\ell=2}^{\infty} k^{-2\ell}\log \ell.\]
 
\subsection{An example mentioned in the Introduction}\label{pelda}
In this subsection, we verify Equation \eqref{peldeq1} and Equation \eqref{peldeq2} from the Introduction.

We set $m_n=\mathbb{P}(o\in U(\mathcal{M}_{G_n}))$. We see that $m_0=1$ and $m_1=0$.
By Remark \ref{remarkmo} and equation~\eqref{zerorek}, for any $n\ge 2$, we have
\[m_n=\frac{1}{1+(2m_{n-2})^{-1}}.\] 

Observe that the equation 
\[x=\frac{1}{1+(2x)^{-1}}\]
has two solutions $x=0$ and $x=\frac{1}{2}$. Thus, a standard argument gives that
\[\lim_{n\to\infty} m_{2n}=\frac{1}{2}\text{ and }\lim_{n\to\infty} m_{2n+1}=0.\]



\bibliography{references}
\bibliographystyle{plain}

\Addresses
\end{document}